\documentclass[12pt]{article}
\usepackage{amsmath}
\usepackage{amssymb}
\usepackage{amsthm}
\usepackage[utf8]{inputenc}
\usepackage[english]{babel}
\usepackage{graphicx}
\usepackage{rotating}
\usepackage{array}
\numberwithin{equation}{section}
\newtheorem{theorem}{Theorem}[section]

\newtheorem{lemma}{Lemma}[section]

\newcommand\ufree{\mathop{\mbox{$u$-$\mathit{free}$}}}

\newcommand\llfree{\mathop{\mbox{$l_1$-$\mathit{free}$}}}
\newcommand\lllfree{\mathop{\mbox{$l_2$-$\mathit{free}$}}}
\newcommand\qfree{\mathop{\mbox{$(q^m-1$)-$\mathit{free}$}}}

\newcommand\F{\mathbb{F}_{q^m}}

\begin{document}
\title{Existence  of Primitive Normal  Pairs with One Prescribed Trace over Finite Fields }
\author{Hariom Sharma, R. K. Sharma}
	
	\date{}
	\maketitle
	\begin{center}
		\textit{ Department of Mathematics, Indian Institute of Technology Delhi, New Delhi, 110016, India}
	\end{center}
\begin{abstract}
	Given $m, n, q\in \mathbb{N}$ such that $q$ is a prime power and $m\geq 3$, $a\in \mathbb{F}_q$, we establish a sufficient condition for the existence of primitive pair $(\alpha, f(\alpha))$ in $\mathbb{F}_{q^m}$ such that $\alpha$ is normal over $\mathbb{F}_q$ and $\text{Tr}_{\F/\mathbb{F}_q}(\alpha^{-1})=a$, where $f(x)\in \mathbb{F}_{q^m}(x)$ is a rational function of degree sum $n$. Further, when $n=2$ and $q=5^k$ for some $k\in \mathbb{N}$, such a pair definitely exists for all $(q, m)$ apart from at most $20$ choices.
	
	\end{abstract}
	
	2010 Math. Sub. Classification: 12E20, 11T23
	\footnote{emails: 
                     hariomsharma638@gmail.com (Hariom), rksharmaiitd@gmail.com (Rajendra)}
	
\section{Introduction}
Given the positive integers $m$ and $ q$ such that $q$ is a prime power, $\mathbb{F}_q$ denotes the finite field of order $q$ and $\mathbb{F}_{q^m}$ be the extension of $\mathbb{F}_q$ of degree $m$.
A generator of the cyclic multiplicative group $\mathbb{F}_{q^m}^*$ is known as a \textit{primitive element} of $\mathbb{F}_{q^m}$. For a rational function $f(x)\in \mathbb{F}_{q^m}(x)$ and $\alpha\in \mathbb{F}_{q^m}$, we call a pair $(\alpha, f(\alpha))$  a \textit{primitive pair} in $\mathbb{F}_{q^m} $ if both $\alpha$ and $f(\alpha)$ are primitive elements of $\mathbb{F}_{q^m}$. Further, $\alpha$ is \textit{normal} over $\mathbb{F}_q$ if the set $\{\alpha, \alpha^q, \alpha^{q^2}, \cdots, \alpha^{q^{m-1}}\}$ forms a basis of $\mathbb{F}_{q^m}$ over $\mathbb{F}_q$. Also, the \textit{trace} of  $\alpha$ over $\mathbb{F}_q$, denoted by $\text{Tr}_{\F/\mathbb{F}_q}(\alpha)$ is given by $\alpha+\alpha^q+\alpha^{q^2}+\cdots+\alpha^{q^{m-1}}$.

Primitive normal elements play a vital role in coding theory and cryptography \cite{application}. Therefore, study of existence of such elements is an active area of research. We refer to \cite{lidl} for the existence of primitive and normal elements in finite fields. Existence of both primitive and normal elements simultaneously was first established by Lenstra and Schoof in \cite{pnbt}. Later on, by using sieving techniques, Cohen and Huczynska \cite{pnbtwc} provided a computer-free proof of it. In 1985, Cohen studied the existence of primitive pair $(\alpha, f(\alpha))$ in $\mathbb{F}_q$ for the rational function $f(x)=x+a, a\in \mathbb{F}_q$. Many more researchers worked in this direction and proved the existence of primitive pair for more general rational function \cite{JNT, special, ambrish, booker}. Additionally, in the fields of even order,  Cohen\cite{even} established the existence of primitive pair $(\alpha, f(\alpha))$ in $\mathbb{F}_{q^n}$ such that $\alpha$ is normal over $\mathbb{F}_q$, where $f(x)=\frac{x^2+1}{x}$. Similar  result has been obtained in \cite{special} for the rational function $f(x)=\frac{ax^2+bx+c}{dx+e}$. Another interesting problem is to prove the existence of primitive pair with prescribed traces which have been discussed in \cite{COMM, FFAAnju,  CommAnju}.

In this article, we consider all the conditions simultaneously and prove the existence of primitive pair $(\alpha, f(\alpha))$ in $\mathbb{F}_{q^m}$ such that $\alpha$ is normal over $\mathbb{F}_q$ and for prescribed $a\in \mathbb{F}_q$, $\text{Tr}_{\F/\mathbb{F}_q}(\alpha^{-1})=a$, where $f(x)$ is more general rational function. To proceed further, we shall use some basic terminology and conventions used in \cite{JNT}.
To say that a non zero polynomial $f(x)\in \F[x]$ has {\em degree} $n \geq 0$ we mean that  $f(x)=a_nx^n+ \cdots+a_0$, where $a_n \neq 0$ and write it as $\deg(f)=n$.  Next, for a rational function $f(x)=f_1(x)/f_2(x)\in \F(x)$, we always assume  that $f_1$ and $f_2$ are coprime and degree sum of $f=\deg(f_1)+\deg(f_2)$. Also, we can divide each of $f_1$ and $f_2$ by the leading coefficient of $f_2$ and suppose that $f_2$ is monic. Further, we say that a rational function $f\in \mathbb{F}_{q^m}(x)$ is exceptional if $f=cx^ig^d$ for some $c\in \F, i\in \mathbb{Z}$(set of integers) and $d>1$ divides $q^m-1$ or $f(x)=x^i$ for some $i \in \mathbb{Z}$ such that $\gcd(q^m-1, i)\neq 1.$

Finally, we introduce some sets which have an important role in this article. For $n_1, n_2\in \mathbb{N}$, $S_{q, m}(n_1, n_2)$ will be used to denote the set of non exceptional rational functions $f=f_1/f_2\in \F(x)$ with $\deg(f_1)\leq n_1$ and $\deg(f_2)\leq n_2$, and $T_{n_1, n_2}$ as the set of pairs $(q,m)\in \mathbb{N}\times \mathbb{N}$ such that for any given $f\in S_{q, m}(n_1, n_2)$ and prescribed $a\in \mathbb{F}_q$, $\F$ contains a normal element  $\alpha$ with $(\alpha, f(\alpha))$ a primitive pair and $\text{Tr}_{\F/\mathbb{F}_q}(\alpha^{-1})=a$. Define $S_{q, m}(n)= \bigcup\limits_{n_1+n_2=n}S_{q,m}(n_1, n_2)$ 
and $T_n=\bigcap\limits_{n_1+n_2=n}T_{n_1,n_2}$. 
By \cite{chou}, for $m\leq 2$, there does not exist any primitive element $\alpha$ such that  $\text{Tr}_{\mathbb{F}_{q^m}/ \mathbb{F}_q}(\alpha^{-1})=0$. Therefore, we shall assume $m\geq 3$   throughout the article.

 In this paper, for $n\in \mathbb{N}$, we take $f(x)\in S_{q, m}(n)$ a  general rational  function of degree sum $n$  and $a\in \mathbb{F}_q$, and prove the existence of normal element $\alpha $ such that $(\alpha, f(\alpha))$ is a primitive pair in $\mathbb{F}_{q^m}$ and $\text{Tr}_{\mathbb{F}_{q^m}/\mathbb{F}_q}(\alpha^{-1})=a$. To be more precise, in section $3$, we obtain a sufficient condition for the existence of such elements in  $\mathbb{F}_{q^m}$. In section $4$, we further improve the condition by proving a generalization of sieving technique due to Anju and Cohen\cite{AnjuCohen}. In section $5$, we demonstrate the application of the results of section $3$ and section $4$ by working with the finite fields of characteristic $5$ and $n=2$. More precisely, we get a subset of $T_2$.

\section{Preliminaries}
In this section, we provide some preliminary notations, definitions and results which are required further in this article. Throughout this article, $m\geq 3$ is an integer, $q$ is an arbitrary prime power and $\mathbb{F}_{q}$ is a finite field of  order $q$. For each $k(>1)\in \mathbb{N}$,  $\omega(k)$  denotes the number of prime divisors of $k$ 
and $W(k)$ denotes the number of square free divisors of $k$. Also for $g(x)\in \mathbb{F}_q[x]$, $\Omega_q(g)$ and $W(g)$ denote the number of monic irreducible(over $\mathbb{F}_q)$ divisors of $g$  and number of square free divisors of $g$ respectively, i.e.,  $W(k)=2^{\omega(k)}$ and $W(g)=2^{\Omega_q(g)}$.

For a finite abelian group  $G$, a homomorphism $\chi$ from $G$ into the multiplicative group $S^1=\{z\in \mathbb{C}: |z|=1\}$ is known as a character of $G$. The set of all characters of $G$ forms a group under multiplication, which is isomorphic to $G$ and is denoted by $\widehat{G}$. Further, the character $\chi_0$, defined as $\chi_0(g) =1$ for all $g\in G$ is called the trivial character of $G$.  The order of a character $\chi$ is the smallest positive integer $r$ such that $\chi^r=\chi_0$. For a finite field $\mathbb{F}_{q^m}$, the characters of the additive group $\mathbb{F}_{q^m}$ and the multiplicative group $\mathbb{F}^*_{q^m}$ are called additive characters and multiplicative characters respectively. A multiplicative character $\chi \in \widehat{\mathbb{F}}_{q^m}^*$ is extended from $\mathbb{F}^*_{q^m}$ to $\mathbb{F}_{q^m}$ by the rule $	
\chi(0)=
\begin{cases}
	0 ~~\text{ if } \chi\neq\chi_0 \\
	1 ~~\text{ if } \chi=\chi_0
\end{cases}.$
For more fundamentals on characters, primitive elements and finite fields, we refer the reader to \cite{lidl}. 

For a divisor $u$ of $q^m-1$, an element $w\in \mathbb{F}_{q^m}^*$ is $\ufree$, if $w=v^d$, where $v\in \mathbb{F}_{q^m}$ and $d|u$ implies $d=1$.
It is easy to observe that an element in $ \mathbb{F}_{q^m}^*$ is $\qfree$ if and only if it is primitive. A special case of \cite[Lemma 10]{shuqin2004character}, provides an interesting result. 
\begin{lemma}
	Let $u$ be a divisor of $q^m-1$, $\xi\in \mathbb{F}_{q^m}^*$. Then
	
	\[
	\sum_{d|u}\frac{\mu(d)}{\phi(d)}\sum_{\chi_d}\chi_d(\xi)=
	\begin{cases}
	\frac{u}{\phi(u)} &\quad\text{ if } \xi \text{ is } \ufree,\\
	0 &\quad\text{otherwise.}
	\end{cases}
	\]
	where $\mu(\cdot)$ is the M$\ddot{\text{o}}$bius function and $\phi(\cdot)$ is the Euler function, $\chi_d$ runs through all
	the $\phi(d)$ multiplicative characters over $\mathbb{F}_{q^m}^*$ with order $d$.
	
\end{lemma}

Therefore, for each divisor $u$ of $q^m-1$, 
\begin{equation}\label{charufree}
\rho_{u} : \alpha \mapsto \theta(u)\sum_{d|u}\frac{\mu(d)}{\phi(d)}\sum_{\chi_d}\chi_d(\alpha),
\end{equation} 

gives a characteristic function for the subset of $\ufree$ elements of $\mathbb{F}_{q^m}^*$, where $\theta(u)=\frac{\phi(u)}{u}$. \\
Also, for each $a\in \mathbb{F}_q$,
\begin{equation*}\tau_a : \alpha \mapsto  \frac{1}{q} \sum \limits_{\psi\in \widehat{\mathbb{F}}_q}\psi(\text{ Tr}_{\mathbb{F}_{q^m}/ \mathbb{F}_q}(\alpha)-a) 
\end{equation*} is a characterstic function for the subset of $\mathbb{F}_{q^m}$ consisting elements  with $\text{ Tr}_{\mathbb{F}_{q^m}/ \mathbb{F}_q}(\alpha)=a$. From \cite[Theorem 5.7]{lidl} every additive character $\psi$ of $\mathbb{F}_q$ can be obtained by $\psi(a)=\psi_0(ua)$, where $\psi_0$ is the canonical additive character of $\mathbb{F}_q$ and $u$ is an element of $\mathbb{F}_q$ corresponding to $\psi$. Thus 

\begin{equation*}
\tau_a(\alpha)=  \frac{1}{q} \sum \limits_{u\in \mathbb{F}_q}\psi_0(\text{ Tr}_{\mathbb{F}_{q^m}/ \mathbb{F}_q}(u\alpha)-ua)  
\end{equation*}

\begin{equation}\label{chartrace}
=\frac{1}{q} \sum \limits_{u\in \mathbb{F}_q}\hat{\psi_0}(u\alpha)\psi_0(-ua),
\end{equation} where $\hat{\psi_0}$ is the additive character of $\mathbb{F}_{q^m}$ defined by $\hat{\psi_0}(\alpha) = \psi_0(\text{ Tr}_{\mathbb{F}_{q^m}/ \mathbb{F}_q}(\alpha))$. In particular, $\hat{\psi_0}$ is the canonical additive character of $\mathbb{F}_{q^m}$.

The additive group of $\mathbb{F}_{q^m}$ is an $\mathbb{F}_q[x]$-module under the rule $f~o~\alpha=\sum\limits_{i=1}^{k}a_i\alpha^{q^i}$; for $\alpha\in \mathbb{F}_{q^m}$ and $f(x)=\sum\limits_{i=1}^ka_ix^i\in \mathbb{F}_q[x]$. For $\alpha\in \mathbb{F}_{q^m}$, the $\mathbb{F}_q$-order of $\alpha$ is the unique monic polynomial $g$ of least degree such that $g~o~\alpha=0$. Observe that $g$ is a factor of $x^m-1$. Similarly, by defining the action of $\mathbb{F}_q[x]$ over $\widehat{\mathbb{F}}_{q^m}$ by the rule $\psi~o~f(\alpha)=\psi(f~o~\alpha)$, where $\psi\in \widehat{\mathbb{F}}_{q^m}, \alpha\in \mathbb{F}_{q^m}$ and $f\in \mathbb{F}_q[x]$,  $\widehat{\mathbb{F}}_{q^m}$ becomes an $\mathbb{F}_q[x]$-module, and the unique monic polynomial $g$ of least degree such that $\psi~o~g=\chi_0$ is called the $\mathbb{F}_q$-order of $\psi$. Further there are $\Phi_q(g)$ characters of $\mathbb{F}_q$-order $g$, where $\Phi_q(g)$  is the analogue of Euler's phi-function on $\mathbb{F}_q[x]$(see \cite{lidl}). 

Similar to  above, for $g|x^m-1$  an element $\alpha\in \mathbb{F}_{q^{m}}$ is $g$-$free$, if $\alpha=h~o~\beta$, where $\beta\in \mathbb{F}_{q^m}$ and $h|g$ implies $h=1.$ It is straightforward that an element in $\mathbb{F}_{q^m}$ is $(x^m-1)$-$free$ if and only if it is normal. Also, for $g|x^m-1$ an expression for the characteristic function for $g$-$free$ elements is given by 
\begin{equation}\label{charfree}
	\kappa_{g} : \alpha \mapsto \Theta(g)\sum_{h|g}\frac{\mu'(d)}{\Phi_q(h)}\sum_{\psi_h}\psi_h(\alpha),
\end{equation} 
where $\Theta(g)=\frac{\Phi_q(g)}{q^{deg(g)}}$, the internal sum runs over all characters $\psi_h$ of $\mathbb{F}_q$-order $h$ and $\mu'$ is the analogue of the M$\ddot{\text{o}}$bius function defined as 

\[
\mu'(g)=
\begin{cases}
	(-1)^s &\text{if } $g$ \text{ is a product of $s$ distinct monic irreducible polynomials},\\
	0 &\quad\text{otherwise.}
\end{cases}
\]

Following results of  D. Wang and L. Fu will play a vital role in our next section.
\begin{lemma}\label{lemma1}$\cite[Theorem~4.5]{PPR4}$ Let $f(x)\in \mathbb{F}_{{q}^d}(x)$ be a rational function. Write $f(x)= \prod_{j=1}^{k}f_j(x)^{n_j}$,
	where $f_j(x)\in \mathbb{F}_{{q}^d}[x]$ are irreducible polynomials and $n_j$ are non zero integers. Let $\chi$ be a multiplicative character of $\mathbb{F}_{q^d}$. Suppose that the rational function $\prod_{i=0}^{d-1}f(x^{q^i})$ is not of the form $h(x)^{\text{ord}(\chi)}$ in $\mathbb{F}_{q^d}(x),$ where ord$(\chi)$ is the order of $\chi$, then we have $$\big|\sum_{\alpha\in \mathbb{F}_{q},f(\alpha)\neq 0, f(\alpha)\neq \infty}\chi(f(\alpha))\big|\leq (d\sum_{j=1}^k \deg(f_j)-1)q^{\frac{1}{2}}.$$
\end{lemma}

\begin{lemma}\label{lemma2}$\cite[Theorem~4.6]{PPR4}$
	Let $f(x), g(x)\in \mathbb{F}_{q^m}(x)$ be rational functions. Write $f(x)= \prod_{j=1}^{k}f_j(x)^{n_j}$,
	where $f_j(x)\in \mathbb{F}_{{q}^m}[x]$ are irreducible polynomials and $n_j$ are non zero integers. Let $D_1=\sum_{j=1}^{k}\deg(f_j)$, let $D_2= max(\deg(g),0)$, let $D_3$ be the degree of denominator of $g(x)$, and let $D_4$ be the sum of degrees of those irreducible polynomials dividing denominator of $g$ but distinct from $f_j(x)(j=1, 2, \cdots,  k)$. Let $\chi$ be a multiplicative character of $\mathbb{F}_{q^m}$, and let $\psi$ be a non trivial additive character of $\mathbb{F}_{q^m}$. Suppose $g(x)$ is not of the  form $r(x)^{q^m}-r(x)$ in $\mathbb{F}_{q^m}(x)$. Then we have the estimate 
	$$\big|\sum_{\alpha\in \mathbb{F}_{q^m},f(\alpha)\neq 0, \infty g(\alpha)\neq \infty}\chi(f(\alpha))\psi(g(\alpha)) \big|\leq (D_1+D_2+D_3+D_4-1)q^{\frac{m}{2}}.$$
\end{lemma}
\section{Sufficient condition}
 Let $l_1,l_2\in \mathbb{N}$ be such that  $l_1,l_2 |q^m-1$. Also, $a\in \mathbb{F}_{q}$, $f(x)\in S_{q,m}(n)$ and $g|x^m-1$, then $N_{f, a, n}(l_1, l_2, g)$ denote the number of elements $\alpha\in \mathbb{F}_{q^m}$ such that $\alpha$ is both $\llfree$ and $g$-$free$, $f(\alpha)$ is $\lllfree$ and $\text{Tr}_{\F/\mathbb{F}_q}(\alpha^{-1})=a$. \\

 We now prove one of the sufficient condition as follows.

\begin{theorem}\label{main}
Let $m, n \text{ and }q\in \mathbb{N}$ such that $q$ is a prime power and  $m\geq 3$.  Suppose that 
\begin{equation} \label{condition}
 q^{\frac{m}{2}-1}>(n+2)W(q-1)^2W(x^m-1).
 \end{equation}
 Then $(q,m)\in T_{n}$.
\end{theorem}

\begin{proof}
 To prove the result, it is enough to show that $N_{f, a, n}(q^m-1, q^m-1, x^m-1)>0$ for every   $f(x)\in S_{q, m}(n)$ and prescribed $a\in \mathbb{F}_q$. Let $f(x)\in S_{q, m}(n)$ be any   rational function and $a\in \mathbb{F}_q$.  Let  $U_1$ be the set of zeros and  poles of $f(x)$ in     $\mathbb{F}_{q^m}$ and $U=U_1\cup\{0\}$.   Assume $l_1, l_2$ be divisors of $q^m-1$ and $g$ be a divisor of $x^m-1$. Then by definition  $$N_{f, a, n}(l_1,l_2, g)=\sum_{\alpha\in \mathbb{F}_{q^m} \setminus U} \rho_{l_1}(\alpha) \rho_{l_2}(f(\alpha))\tau_a(\alpha^{-1})\kappa_g(\alpha)$$ now using (\ref{charufree}), (\ref{chartrace}) and (\ref{charfree}),  
 \begin{align}\label{key}
 	N_{f,a,n}(l_1,l_2,g)=\frac{\theta(l_1)\theta(l_2)\Theta(g)}{q} \sum\limits_{\substack{d_1|l_1, d_2|l_2\\h|g }} \frac{\mu(d_1)}{\phi(d_1)} \frac{\mu(d_2)}{\phi(d_2)}\frac{\mu'(h)}{\Phi_q(h)} \sum\limits_{\chi_{d_1}, \chi_{d_2}, \psi_h}\chi_{f,a}(d_1, d_2, h),
  \end{align}
where $\chi_{f,a}(d_1, d_2, h)=\sum\limits_{u\in \mathbb{F}_q}\psi_0(-au) \sum\limits_{\alpha\in \mathbb{F}_{q^m}\setminus U}\chi_{d_1}(\alpha)\chi_{d_2}(f(\alpha))\psi_h(\alpha)\hat{\psi_0}(u\alpha^{-1})$.\\
Since $\psi_h$ is an additive character of $\mathbb{F}_{q^m}$ and $\hat{\psi_0}$ is canonical additive character of $\mathbb{F}_{q^m}$, therefore there exists $v\in \mathbb{F}_{q^m}$ such that $\psi_h(\alpha)=\hat{\psi_0}(v\alpha)$. Hence 
  $\chi_{f,a}(d_1, d_2, h)=\sum\limits_{u\in \mathbb{F}_q}\psi_0(-au) \sum\limits_{\alpha\in \mathbb{F}_{q^m}\setminus U}\chi_{d_1}(\alpha)\chi_{d_2}(f(\alpha))\hat{\psi_0}(v\alpha+u\alpha^{-1})$. 
  
  At this point, we claim that if $(d_1, d_2, h)\neq (1,1,1)$, where third $1$ denotes the unity of $\mathbb{F}_q[x]$, then $|\chi_{f,a}(d_1, d_2, h)|\leq (n+2)q^{\frac{m}{2}+1}$. To see the claim, first suppose $d_2=1$, then $\chi_{f,a}(d_1, d_2, h)=\sum\limits_{u\in \mathbb{F}_q}\psi_0(-au) \sum\limits_{\alpha\in \mathbb{F}_{q^m}\setminus U}\chi_{d_1}(\alpha)\hat{\psi_0}(v\alpha+u\alpha^{-1})$. Here, if $vx+ux^{-1}\neq r(x)^{q^m}-r(x)$ for any $r(x)\in \mathbb{F}_{q^m}(x)$ then by Lemma \ref{lemma2} $|\chi_{f, a}(d_1,d_2, h)|\leq 2q^{\frac{m}{2}+1}+(|U|-1)q\leq (n+2)q^{\frac{m}{2}+1}$. Also, if $vx+ux^{-1}=r(x)^{q^m}-r(x)$ for some $r(x)\in \mathbb{F}_{q^m}(x)$ then following [Comm. Anju], it is possible when $u=v=0$, which implies, $|\chi_{f, a}(d_1, d_2, h)|\leq |U|q<(n+2)q^{\frac{m}{2}+1}$.
  
  Now suppose $d_2>1$. Let $d$ be the least common multiple of $d_1$ and $d_2$. Then \cite{lidl} suggests that there exists a character  $\chi_d$ of order $d$ such that $\chi_{d_2}=\chi_d^{d/d_2}$. Also, there is an integer $0\leq k<q^m-1$ such that $\chi_{d_1}=\chi_d^k$. Consequently, $\chi_{f,a}(d_1, d_2, h)=\sum\limits_{u\in \mathbb{F}_q}\psi_0(-au) \sum\limits_{\alpha\in \mathbb{F}_{q^m}\setminus U}\chi_{d}(\alpha^kf(\alpha)^{d/d_2})\hat{\psi_0}(v\alpha+u\alpha^{-1})$. At this moment, first suppose $vx+ux^{-1}\neq r(x)^{q^m}-r(x)$ for any $r(x)\in \mathbb{F}_{q^m}(x)$. Then Lemma \ref{lemma2} implies that $|\chi_{f,a}(d_1, d_2, h)|\leq (n+2)q^{\frac{m}{2}+1}$. Also, if $vx+ux^{-1}= r(x)^{q^m}-r(x)$ for some $r(x)\in \mathbb{F}_{q^m}(x)$, then following \cite{CommAnju} we get $u=v=0$. Therefore,  $\chi_{f,a}(d_1, d_2, h)=\sum\limits_{u\in \mathbb{F}_q}\psi_0(-au) \sum\limits_{\alpha\in \mathbb{F}_{q^m}\setminus U}\chi_{d}(\alpha^kf(\alpha)^{d/d_2})$. Here, if $x^kf(x)^{d/d_2}\neq r(x)^d$ for any $r(x)\in \mathbb{F}_{q^m}(x)$, then using Lemma \ref{lemma1} we get $|\chi_{f,a}(d_1, d_2, h)|\leq nq^{\frac{m}{2}+1}<(n+2)q^{\frac{m}{2}+1}$. However, $x^kf(x)^{d/d_2}= r(x)^d$ for some $r(x)\in \mathbb{F}_{q^m}(x)$ gives that $f$ is exceptional(see \cite{JNT}).

Hence, from the above discussion along with (\ref{key}), we get

 \begin{align*}
 	N_{f,a,n}(l_1, l_2, g)\geq \frac{\theta(l_1) \theta(l_2)\Theta(g)}{q}(q^m-|U|-((n+2)q^{\frac{m}{2}+1})(W(l_1)W(l_2)W(g)-1))\\
 	\geq  \frac{\theta(l_1) \theta(l_2)\Theta(g)}{q}(q^m-(n+1)-((n+2)q^{\frac{m}{2}+1})(W(l_1)W(l_2)W(g)-1))
 \end{align*}

 \begin{align}\label{3.3}
\geq  \frac{\theta(l_1) \theta(l_2)\Theta(g)}{q}(q^m-(n+2)q^{\frac{m}{2}+1}W(l_1)W(l_2)W(g))
 \end{align}
 
 Thus, if $q^{\frac{m}{2}-1}>(n+2)W(l_1)W(l_2)W(g)$, then $N_{f,a,n}(l_1, l_2, g)>0$ for all $f(x)\in S_q(n)$ and prescribed $a\in \mathbb{F}_{q}$. The result now follows by taking $l_1=l_2=q^m-1$ and $g=x^m-1$.
\end{proof}
\section{Sieving Results}
Here,  we state some results, their proofs have been omitted as they follow on the lines of the results in \cite{FFAAnju} and have been used frequently in \cite{COMM, JNT, FFAAnju, ambrish, special}. 
\begin{lemma}\label{initial1}
		Let $k \text{ and } P$ be co-prime positive integers and $g, G\in \mathbb{F}_q[x]$ be co-prime polynomials. Also, let $\{p_1, p_2, \cdots, p_r\}$ be the collection of all prime divisors of $P$, and $\{g_1, g_2, \cdots, g_s\}$ contains all the irreducible factors of $G$. Then
\begin{align*}
		N_{f, a, n}(kP, kP, gG)\geq \sum\limits_{i=1}^rN_{f,a,n}(kp_i, k,g )+\sum\limits_{i=1}^rN_{f,a,n}(k, kp_i,g )\\+\sum\limits_{i=1}^sN_{f,a,n}(k, k,gg_i )-(2r+s-1)N_{f,a,n}(k,k,g).
	\end{align*}
\end{lemma}
\begin{lemma}\label{initial2} 	Let $l, m, q\in \mathbb{N}$, $g\in \mathbb{F}_q[x]$ be such that $q$ is a prime power, $m\geq 3$ and  $l|q^m-1$,   $g|x^m-1$. Let $c$ be a prime number which divides $q^m-1$ but not $l$, and $e$ be irreducible polynomial dividing $x^m-1$ but not $g$. Then
	\begin{align*}
		|N_{f,a,n}(cl,l,g)-\theta(c)N_{f,a,n}(l,l,g)|\leq (n+2)\theta(c)\theta(l)^2\Theta(g)W(l)^2W(g)q^{\frac{m}{2}},
	\end{align*}
	\begin{align*}
	|N_{f,a,n}(l,cl,g)-\theta(c)N_{f,a,n}(l,l,g)|\leq (n+2)\theta(c)\theta(l)^2\Theta(g)W(l)^2W(g)q^{\frac{m}{2}}
\end{align*}
and 
	\begin{align*}
	|N_{f,a,n}(l,l,eg)-\Theta(e)N_{f,a,n}(l,l,g)|\leq (n+2)\theta(l)^2\Theta(e)\Theta(g)W(l)^2W(g)q^{\frac{m}{2}}.
\end{align*}
\end{lemma}
\begin{theorem}\label{PSC}
		Let $l, m, q\in \mathbb{N}$, $g\in \mathbb{F}_q[x]$ be such that $q$ is a prime power, $m\geq 3$ and  $l|q^m-1$,   $g|x^m-1$. Also, let $\{p_1, p_2, \cdots p_r\}$ be the collection of primes which divides $q^m-1$ but not $l$, and $\{g_1, g_2, \cdots g_s\}$ be the irreducible polynomials dividing $x^m-1$ but not $g$. Suppose $\delta=1-2\sum\limits_{i=1}^r\frac{1}{p_i}-\sum\limits_{i=1}^s\frac{1}{q^{\deg(g_i)}}, \delta>0$ and $\Delta=\frac{2r+s-1}{\delta}+2$. If $q^{\frac{m}{2}-1}>(n+2)\Delta W(l)^2W(g)$ then $(q,m)\in T_n.$
\end{theorem}

Now, we present a more effective sieving technique than Theorem \ref{PSC}, which is an extension of the result in \cite{AnjuCohen}. For this, we adopt some notations and conventions from \cite{AnjuCohen} as described. Let $\text{Rad}(q^m-1)=kPL$, where $k$ is the product of smallest prime divisors of $q^m-1$, $L$ is the product of large prime divisors of $q^m-1$ denoted by $L=l_1\cdot l_2 \cdots l_t$, and rest of the prime divisors of $q^m-1$ lie in $P$ and  denoted by $p_1, p_2, \cdots, p_r$. Similarly, $\text{Rad}(x^m-1)=gGH$, where $g$ is the product of irreducible factors of $x^m-1$ of least degree, and irreducible factors of large degree are factors of $H$ which are denoted by $h_1, h_2, \cdots, h_u$ and rest lie in $G$ and denoted by $g_1, g_2, \cdots, g_s$. 
\begin{theorem}\label{MPSC} Let $m, q\in \mathbb{N}$ such that $q$ is a prime power and $m\geq 3$. Using above notations, let $\text{Rad}(q^m-1)=kPL$, $\text{Rad}(x^m-1)=gGH$,  $\delta=1-2\sum\limits_{i=1}^r\frac{1}{p_i}-\sum\limits_{i=1}^s\frac{1}{q^{\deg(g_i)}}, \epsilon_1=\sum\limits_{i=1}^t\frac{1}{l_i},~ \epsilon_2=\sum\limits_{i=1}^u\frac{1}{q^{\deg(h_i)}} \text{ and } \delta\theta(k)^2\Theta(g)-(2\epsilon_1+\epsilon_2)>0$. Then \begin{multline}\label{mpsc}
	q^{\frac{m}{2}-1}>(n+2)[\theta(k)^2\Theta(g)W(k)^2W(g)(2r+s-1+2\delta)+(t-\epsilon_1)+(2/(n+2))(u-\epsilon_2)\\+(n/(n+2))(1/q^{m/2})(t+u-\epsilon_1-\epsilon_2)]/[\delta\theta(k)^2\Theta(g)-(2\epsilon_1+\epsilon_2)]\end{multline}
	implies $(q,m)\in T_n$.  
	\end{theorem}
\begin{proof}
	Clearly, \begin{multline}\label{4}
		 N_{f,a,n}(q^m-1, q^m-1, x^m-1)=N_{f,a,n}(kPL, kPL, gGH)\geq N_{f,a,n}(kP, kP, gG)\\+N_{f,a,n}(L, L, H)-N_{f,a,n}(1, 1, 1).
	\end{multline} Further, by Lemma \ref{initial1} 	\begin{multline*}
		N_{f,a,n}(kP, kP, gG)\geq \delta N_{f,a,n}(k,k,g)+ \sum\limits_{i=1}^r\{N_{f,a,n}(kp_i, k,g )-\theta(p_i)N_{f,a,n}(k,k,g)\}\\+\sum\limits_{i=1}^r\{N_{f,a,n}(k, kp_i,g )-\theta(p_i)N_{f,a,n}(k,k,g)\}+\sum\limits_{i=1}^s(N_{f,a,n}(k, k,gg_i )-\Theta(g_i)N_{f,a,n}(k,k,g))
\end{multline*}.
		 
		Using (\ref{3.3}) and Lemma \ref{initial2}, we get
	 \begin{align*}
	 	N_{f,a,n}(kP, kP, gG)\geq \delta\theta(k)^2\Theta(g)\big(q^{m-1}-(n+2)W(k)^2W(g)q^\frac{m}{2}\big)\\-(n+2)\theta(k)^2\Theta(g)W(k)^2W(g)\big(\sum\limits_{i=1}^r2\theta(p_i)+\sum\limits_{i=1}^s\Theta(g_i)\big)q^{\frac{m}{2}}
	 \end{align*}
 \begin{align}\label{4.1}
 	=\theta(k)^2\Theta(g)\big(\delta q^{m-1}-(n+2)(2r+s-1+2\delta)W(k)^2W(g)q^{\frac{m}{2}}\big).
 \end{align} Again, by Lemma \ref{initial1}
\begin{multline*}
	N_{f,a,n}(L, L, H)-N_{f,a,n}(1, 1, 1)\geq \sum\limits_{i=1}^tN_{f,a,n}(l_i, 1, 1)+\sum\limits_{i=1}^tN_{f,a,n}(1, l_i, 1)\\+\sum\limits_{i=1}^uN_{f,a,n}(1, 1, h_i)-(2t+u)N_{f,a,n}(1, 1, 1)
\end{multline*}
\begin{multline}\label{4.2}
		=\sum\limits_{i=1}^t\{N_{f,a,n}(l_i, 1, 1)-\theta(l_i)N_{f,a,n}(1, 1, 1)\}+\sum\limits_{i=1}^t\{N_{f,a,n}(1, l_i, 1)-\theta(l_i)N_{f,a,n}(1, 1, 1)\}\\	+\sum\limits_{i=1}^u\{N_{f,a,n}(1, 1, h_i)-\Theta(h_i)N_{f,a,n}(1, 1, 1)\}-(2\epsilon_1+\epsilon_2)N_{f,a,n}(1, 1, 1)
\end{multline}
By (\ref{key}), for a prime divisor $l$ of $q^m-1$,
$|N_{f,a,n}(l, 1, 1)-\theta(l)N_{f,a,n}(1, 1, 1)|=\frac{\theta(l)}{\phi(l)q}|\sum\limits_{\chi_l}\chi_{f,a}(l, 1, 1)|,$ where
\begin{align*}
	|\chi_{f,a}(l, 1, 1)|=|\sum\limits_{u\in \mathbb{F}_q}\psi_0(-au)\sum\limits_{\alpha\in \mathbb{F}_{q^m}\setminus U}\chi_l(\alpha)\hat{\psi_0}(u\alpha^{-1}|\leq q^{\frac{m}{2}+1}+nq.
\end{align*} Hence, $|N_{f,a,n}(l, 1, 1)-\theta(l)N_{f,a,n}(1, 1, 1)|\leq \theta(l)(q^{\frac{m}{2}}+n).$
Similarly, 
\begin{align*}
	|\chi_{f,a}(1, l, 1)|=|\sum\limits_{u\in \mathbb{F}_q}\psi_0(-au)\sum\limits_{\alpha\in \mathbb{F}_{q^m}\setminus U}\chi_l(f(\alpha))\hat{\psi_0}(u\alpha^{-1}|\leq (n+1)q^{\frac{m}{2}+1},
\end{align*} which further implies $|N_{f,a,n}(1, l, 1)-\theta(l)N_{f,a,n}(1, 1, 1)|\leq (n+1)q^{\frac{m}{2}}$.\\ Also, for an irreducible divisor $h$ of $x^m-1$, 
\begin{multline*}
	|\chi_{f,a}(1, 1, h)|=|\sum\limits_{u\in \mathbb{F}_q}\psi_0(-au)\sum\limits_{\alpha\in \mathbb{F}_{q^m}\setminus U}\psi_h(\alpha)\hat{\psi_0}(u\alpha^{-1}|\\=|\sum\limits_{u\in \mathbb{F}_q}\psi_0(-au)\sum\limits_{\alpha\in \mathbb{F}_{q^m}\setminus U}\hat{\psi_0}(v\alpha+u\alpha^{-1}|\leq 2q^{\frac{m}{2}+1}+nq.
\end{multline*} Therefore,
$|N_{f,a,n}(1, 1, h)-\Theta(h)N_{f,a,n}(1, 1, 1)|\leq \Theta(h)(q^\frac{m}{2}+n)$.  Using these bounds in (\ref{4.2}), we have $	N_{f,a,n}(L, L, H)-N_{f,a,n}(1, 1, 1)\geq -\sum\limits_{i=1}^t \theta(l_i)(q^{\frac{m}{2}}+n)-\sum\limits_{i=1}^t\theta(l_i)(n+1)q^{\frac{m}{2}}-\sum\limits_{i=1}^u\Theta(h_i)(2q^{\frac{m}{2}}+n)-(2t+u)N_{f,a,n}(1, 1, 1)$.
 Now,  $N_{f,a, n}(1, 1, 1) \leq q^{m-1}$ together with $\sum\limits_{i=1}^t\theta(l_i)=(t-\epsilon_1)$ and $\sum\limits_{i=1}^u=(u-\epsilon_2)$ implies
 \begin{multline}\label{4.3}
	N_{f,a,n}(L, L, H)-N_{f,a,n}(1, 1, 1)\geq -\{(n+2)(t-\epsilon_1)+2(u-\epsilon_2)\}q^{\frac{m}{2}}\\
	-n(t+u-\epsilon_1-\epsilon_2)-(2\epsilon_1+\epsilon_2)q^{m-1}.
 \end{multline} 
Now using (\ref{4.1}) and (\ref{4.3}) in (\ref{4}) we get,\begin{multline*}
	 N_{f,a,n}(q^m-1, q^m-1, x^m-1)\geq \{\delta\theta(k)^2\Theta(g)-(2\epsilon_1+\epsilon_2)\}q^{m-1}-\theta(k)^2\Theta(g)(n+2)\\(2r+s-1+2\delta)W(k)^2W(g)q^{\frac{m}{2}}-\{(n+2)(t-\epsilon_1)+2(u-\epsilon_2)\}q^{\frac{m}{2}}-n(t+u-\epsilon_1-\epsilon_2)\\\\
	 =q^\frac{m}{2}\big[\big(\delta\theta(k)^2\Theta(g)-(2\epsilon_1+\epsilon_2)\big)q^{\frac{m}{2}-1}-(n+2)\{\theta(k)^2\Theta(g)(2r+s-1+2\delta)W(k)^2W(g)\\-\{(t-\epsilon_1)+(2/(n+2))(u-\epsilon_2)\}-(n/(n+2))(1/q^{m/2})(t+u-\epsilon_1-\epsilon_2)\}\big]
	\end{multline*}
Thus \begin{multline*}
	q^{\frac{m}{2}-1}>(n+2)[\theta(k)^2\Theta(g)W(k)^2W(g)(2r+s-1+2\delta)+(t-\epsilon_1)+(2/(n+2))(u-\epsilon_2)\\+(n/(n+2))(1/q^{m/2})(t+u-\epsilon_1-\epsilon_2)]/[\delta\theta(k)^2\Theta(g)-(2\epsilon_1+\epsilon_2)]\end{multline*} implies $N_{f,a,n}(q^m-1, q^m-1, x^m-1)>0$ i.e., $(q, m)\in T_n$.

\end{proof}
It is easy to observe that Theorem \ref{PSC} is a special case of Theorem \ref{MPSC} and  can be obtained  by setting $t=u=\epsilon_1=\epsilon_2=0$.
\section{Working Example}
 However the results discussed above are applicable for arbitrary natural number $n$ and  the finite field $\mathbb{F}_{q^m}$ of any prime characteristic. Though to demonstrate the application of above results  and make the calculations uncomplicated we assume  that $q=5^k$ for some $k\in \mathbb{N}$ and $n=2$, and work on the set $T_2$. Precisely, in this section, we prove the following result.
 \begin{theorem}\label{example}
 	Let $q = 5^k$ for some $k\in \mathbb{N}$ and $m\geq 3$ is an integer. Then
 	$(q,m)\in T_2$ unless one of the following holds:
 	\begin{enumerate}
 		\item $q = 5, 5^2, 5^3, 5^4, 5^5, 5^6, 5^8, 5^{10}$ and $m = 3$;
 		\item $q = 5, 5^2, 5^3, 5^4$ and $m = 4$;
 		\item $q = 5, 5^2$ and $m = 5, 6;$
 		\item $q = 5$ and $m = 7, 8, 10, 12.$
 	\end{enumerate}
 \end{theorem}
 
We shall divide it in two parts, in first part we shall work on $m\geq5$ and in second we shall consider $m=3, 4$. For further calculation work and to apply the previous results we shall need the following lemma which can also be developed from \cite[Lemma 6.2]{even}.\\
\begin{lemma}\label{boundlemma}
	Let $M$ be a positive integer, then $W(M)<4515\times M^{1/8}.$
\end{lemma}
\subsection{Part 1.} In this part, we assume $m\geq 5$ and  write $m=m'5^j$, where $j\geq 1$ is an integer and $5\nmid m'$. Then $\Omega_q(x^m-1)=\Omega_q(x^{m'}-1)$ which further implies $W(x^m-1)=W(x^{m'}-1)$. Further, we shall  divide the discussion in two cases.\\
$\bullet~~~ m'|q-1$\\
$\bullet~~~ m'\nmid q-1$\\\\
\textbf{Case 1.} $m|q-1$.\\
Clearly \cite[Theorem 2.47]{lidl} implies that $\Omega_q(x^{m'}-1)=m'$.  Let $l=q^m-1 \text{ and } g=1$ in Theorem \ref{PSC} then $\Delta=\frac{q^2+(a-3)q+2}{(a-1)q+1}$, where $a=\frac{q-1}{m'},$ which further implies $\Delta<q^2$. Hence $(q, m)\in T_2$ if $q^{\frac{m}{2}-3}>4W(q^m-1)^2.$ However, by Lemma \ref{boundlemma}, it is sufficient if $q^{\frac{m}{4}-3}>4\cdot(4515)^2,$ which holds for $q\geq 125$ and for all $m\geq 28$. In particular, for $q\geq 125$ and for all $m'\geq 28$. Next, we examine all the cases where $m'\leq 27$. For this we set $l=q^m-1$ and $g=1$ in Theorem \ref{PSC} unless mentioned. Then $\delta=1-\frac{m'}{q}$ and  $\Delta=2+\frac{(m'-1)q}{q-m'}$\\
\textbf{1. } \underline{$m'=1.$}  Here $m=5^j$ for some integer $j\geq 1$ and $\Delta=2$. Then by Theorem \ref{PSC} it is sufficient if $q^{\frac{m}{2}-1}>4\cdot 2\cdot W(q^m-1)^2$. Again Lemma \ref{boundlemma} implies $(q,m)\in T_2$ if $q^{\frac{m}{4}-1}>8\cdot (4515)^2$ i.e., $q^{\frac{5^j}{4}-1}>8\cdot (4515)^2$, which holds for all choices of $(q,m)$ except $(5, 5), (5, 5^2), (5^2, 5), (5^2, 5^2), (5^3, 5), (5^4, 5), ~ \cdots, (5^{46}, 5) $ which are $48$ in number. For these, we checked  $q^{\frac{m}{2}-1}>4\cdot 2\cdot W(q^m-1)^2$ directly by factoring $q^m-1$ and got it verified except the pairs $(5,5), (5^2, 5), (5^3, 5),\\ (5^4, 5)$ and $(5^6, 5)$.\\\\
\textbf{2. }\underline{$m'=2$.} In this case, $m=2\cdot m^j$ for some $j\geq 1$ and $\Delta=2+\frac{q}{q-2}<4$. Similar to the above case, it is sufficient if $q^{\frac{2\cdot 5^j}{4}-1}>16\cdot (4515)^2$, which is true except the $9$ pairs $(5, 10), (5, 50), (5^2, 10), (5^3, 10), \cdots, (5^8, 10)$, and the verification of $q^{\frac{m}{2}-1}>4\cdot 4\cdot W(q^m-1)^2$  for these pairs yield the only possible exceptions as $(5, 10)\text{ and } (5^2, 10)$. 

Following the similar steps for the rest of the values of $m'\leq 27$ we get that there is no exception for many values of $m'$. Values of $m'$ with possible exceptional pairs is as below. \\
\textbf{3. } \underline{$m'=4.$} $(5, 20)$.\\
  \textbf{4. } \underline{$m'=6.$} $(5^2, 6), (5^4, 6) \text{ and }(5^6, 6).$\\
  \textbf{5. } \underline{$m'=8.$} $(5^2, 8)$.
  
  Furthermore, for the pairs $(5^3, 5), (5^4, 5), (5^6, 5), (5^2, 10), (5, 20), (5^4, 6), (5^6, 6)$ and $(5^2, 8)$ Theorem \ref{PSC} holds for some choice of $l$ and $g$ (see Table 1). Hence, only left \textbf{possible exceptions} in this case are $(5,5), (5^2, 5), (5, 10)$ and $(5^2, 6)$.
  \begin{center}
  	Table 1
  	\begin{tabular}{|m{.4cm}|m{1.2cm}|m{.4cm}|m{.4cm}|m{2cm}|m{.4cm} |m{1.5cm}|m{1.7cm}|m{1.8cm}|}
  		\hline 		Sr. No.    &  $(q, m)$ & $l$   &  $r$   & $g$ & $s$ & $\delta>$& $\Delta< $ & $4\Delta W(g) $ $W(l)^2<$ \\
  		\hline
  		1      &  $(5^3, 5)$ & $2$ & $5$  & $1$ & $1$ & $0.705298$ & $16.178405$ & $518$  \\
  		\hline
  		2& $(5^4, 5)$&6 & 6 & $1$ & 1 & $0.581729$ & $22.628164$& $2897$\\
  		\hline 
  		3 & $(5^6,5)$&6 & 9 & $1$ & 1 & 0.390631 & 48.079201& 6155\\
  		\hline
  		 4 & $(5^2,10)$&6 & 6 & $1$ & 2 & 0.503329 & 27.828038& 3562\\
  		 \hline
  		 5 & $(5,20)$&6 & 6 & $x^2+\beta^3x+\beta$ & 2 & 0.183329 & 72.910743& 18666\\
  		\hline
  		6 & $(5^4,6)$&6 & 6 & $1$ & 6 & 0.476599 & 37.669274& 4822\\ 
  		\hline
  		7 & $(5^6,6)$&6 & 9 & $1$ & 6 & 0.330094 & 71.677019& 9175\\ 
  		\hline
  		8 & $(5^2,8)$&6 & 4 & $1$ & 8 & 0.401942 & 39.318735& 5033\\ \hline

  	\end{tabular}
  \end{center} where $\beta$ is a primitive element of $\mathbb{F}_5$.\\\\
   \textbf{Case 2.} $m'\nmid q-1$.\\
 Let the order of $q\mod m'$ be denoted by $b$. Then $b\geq 2$ and degree of irreducible factors of $x^{m'}-1$ over $\mathbb{F}_q$ is less than or equal to $b$. Let $M$ denotes the number of distinct irreducible factors of $x^{m}-1$ over $\mathbb{F}_q$ of degree less than $b$. Also let $\nu(q, m)$ denotes the ratio $\nu(q, m)=\frac{M}{m}$. Then, $m\nu(q ,m)=m'\nu(q, m')$. 
 
 For the further progress, we need the following two results which are the directly implied by Proposition $5.3$ of \cite{pnbtwc} and Lemma 7.2 of \cite{even} respectively.
 \begin{lemma}\label{5.2}
 	Let $k, m, q\in \mathbb{N}$ be such that $q=5^k$  and $m'\nmid q-1.$ In the notations of Theorem \ref{PSC}, let $l=q^m-1$ and $g$ is the product of irreducible factors of $x^m-1$ of degree less than $b$, then $\Delta<m'$.
 \end{lemma}
\begin{lemma}\label{5.3}
	Let $m'>4$ and $m_1=\gcd(q-1, m')$. Then following bounds hold.
	\begin{enumerate}
		\item For $m'=2m_1$, $\nu(q, m')=\frac{1}{2};$
		\item for $m'=4m_1, \nu(q, m')=\frac{3}{8};$
		\item for $m'=6m_1, \nu(q, m')=\frac{13}{36};$
		\item otherwise, $\nu(q, m')\leq \frac{1}{3}$.
	\end{enumerate}

\end{lemma}
At this point we note that $m'=1 ,2$ and $4$ divide $q-1$ for any $q=5^k$ and have been discussed in above case,  whereas $m'=5$ is not possible. Therefore, in this case we need to discuss $m'=3$ and $m'\geq 6$.

First consider $m'=3$. Then $m=3\cdot 5^j$ for some integer $j\geq 1$. Also, $m'\nmid q-1$ implies if $q=5^k$ then $k$ is odd and $x^{m'}-1$ is the product of a linear factor and a quadratic factor. Thus, $W(x^m-1)=W(x^{m'}-1)=2^2=4$ and $(\ref{main})$ implies $(q, m)\in T_2$ if $q^{\frac{m}{2}-1}>16\cdot W(q^m-1)^2.$ By Lemma \ref{boundlemma}, it is sufficient if $q^{\frac{m}{4}-1}>16\cdot (4515)^2$, which hold for $q=5$ and $m\geq 53$, $q=125$ and $m\geq 21$, $q\geq 5^5$ and $m\geq 14$. Thus, only possible exceptions are $(5, 15)$ and $(125, 15)$. For these two possible exceptions we checked $q^{\frac{m}{4}-1}>16\cdot W(q^m-1)^2$ directly by factoring $q^m-1$ and got it verified for $(125, 15)$. Hence only possible exception for $m'=3$ is $(5, 15)$.

Now suppose $m'\geq 6$. At this point, in Theorem \ref{PSC} let $l=q^m-1$ and $g$ be the product of irreducible factors of $x^m-1$ of degree less than $b$. Therefore,  Lemma \ref{5.2} along with Theorem \ref{PSC} implies $(q, m)\in T_2$ if $q^{\frac{m}{2}-1}>4\cdot m'\cdot W(q^m-1)^2\cdot 2^{m'\nu(q, m')}$. By Lemma \ref{boundlemma}, it is sufficient if \begin{align}\label{5.1}
	q^{\frac{m}{4}-1}>4\cdot m\cdot (4515)^2\cdot 2^{m\nu(q, m')}.
\end{align} Further, we shall discuss it in four cases as follows.\\
\textbf{1.} \underline{$m'\neq 2m_1, 4m_1, 6m_1.$}\\
Here, Lemma \ref{5.3} implies $\nu(q, m')=\frac{1}{3}$. Using this in (\ref{5.1}) we get $(q, m)\in T_2$ if $q^{\frac{m}{4}-1}>4\cdot m\cdot (4515)^2\cdot 2^{\frac{m}{3}}$, which holds for $q^m\geq 5^{145}$. Next, for $q^m\leq 5^{144}$, we verified $q^{\frac{m}{2}-1}>4\cdot m\cdot W(q^m-1)^2\cdot 2^{\frac{m}{3}}$ by factoring $q^m-1$ and got a list of $20$ possible exception as follows.\\
$(5, 6), (5, 7), (5, 9), (5, 11), (5, 12), (5,13), (5, 14), (5, 17), (5,18), (5, 19), (5, 21),\\ (5, 22), (5, 27), (5, 30), (5, 36), (5^2, 7), (5^2, 9), (5^2, 11), (5^3, 6), (5^5, 6)$.\\\\
\textbf{2.} \underline{$m'= 2m_1.$}\\
In this case, $\nu(q, m)=\frac{1}{2}$. Therefore, (\ref{5.1}) implies $(q,m)\in T_2$ if $q^{\frac{m}{4}-1}>4\cdot m\cdot (4515)^2\cdot 2^{\frac{m}{2}}$, which holds for $q=5$ and $m\geq 466$ while for $q\geq 25$ it is sufficient that $m\geq 56$. Here, for $q=5$, we have $m'=8$ only. Thus possible exception  for $q=5$ are $(5, 8), (5, 40)$ and $(5, 200)$. On the  other hand,  for $q\geq 25$ and $q^{m}< 25^{56}$ along with above three possible exceptions we checked $q^{\frac{m}{2}-1}>4\cdot m\cdot W(q^m-1)^2\cdot 2^{\frac{m}{2}}$ and got it verified except $(5, 8), (5, 40) \text{ and }(5^3, 8)$.\\\\
\textbf{3.} \underline{$m'= 4m_1.$}\\
Here, $\nu(q, m)=\frac{3}{8}$. Again, (\ref{5.1}) gives $(q,m)\in T_2$ if  $q^{\frac{m}{4}-1}>4\cdot m\cdot (4515)^2\cdot 2^{\frac{3m}{8}}$, which is true for $q^{m}\geq 5^{176}$. On the other side, verification of $q^{\frac{m}{2}-1}>4\cdot m\cdot W(q^m-1)^2\cdot 2^{\frac{3m}{8}}$ for $q^{m}<5^{176}$ provides only possible exception as $(5, 16)$.\\\\
\textbf{4.} \underline{$m'= 6m_1.$}\\
Similar to the above case, we have $\nu(q, m)=\frac{13}{36}$ and $q^{\frac{m}{4}-1}>4\cdot m\cdot (4515)^2\cdot 2^{\frac{13m}{36}}$ holds for $q^m\geq 5^{164}$. Also, for $q^m<5^{164}$, $q^{\frac{m}{2}-1}>4\cdot m\cdot W(q^m-1)^2\cdot 2^{\frac{13m}{36}}$ holds for all $(q, m)$ except $(5, 24)$.

  \begin{center}
	Table 2
	\begin{tabular}{|m{.4cm}|m{1.2cm}|m{.4cm}|m{.4cm}|m{2cm}|m{.4cm} |m{1.5cm}|m{1.7cm}|m{1.8cm}|}
		\hline 		Sr. No.    &  $(q, m)$ & $l$   &  $r$   & $g$ & $s$ & $\delta>$& $\Delta< $ & $4\Delta W(g) $ $W(l)^2<$ \\
		\hline
	1 & $(5,11)$&2 & 1 & $1$& 3 & 0.799359 & 7.004009& 225\\
	 \hline
	 2 & $(5,13)$&2 & 1 & $1$ & 4 & 0.795199 & 8.287731& 266\\ \hline
	 3 & $(5,14)$&2 & 4 & $x+1$ & 3 & 0.059683& 169.55170& 5426\\ \hline
	 4 & $(5,17)$&2 & 2 & $1$ & 2 & 0.795110& 8.288442& 266\\ \hline
	 5 & $(5,18)$&6 & 5 & $1$ & 6 & 0.061578 & 245.59029& 31436\\ \hline
	 6 & $(5,19)$&2 & 3 & $1$ & 3 & 0.789208 & 12.136745& 389\\ \hline
	7 & $(5,21)$&2 & 4 & $1$ & 5 & 0.689908 & 19.393614& 621\\ \hline
	 8 & $(5,22)$&2 & 5 & $x+1$& 5 & 0.014867& 943.67119& 30198\\ \hline
	 9 & $(5,27)$&2 & 7 &$1$ & 4 & 0.561470& 32.277659& 1033\\ \hline
	 10 & $(5,30)$&6 & 9 &$ x+1$ & 3 & 0.110695 & 182.67531& 23383\\ \hline
	 11 & $(5,36)$&6 & 9 & $x^4-1$ & 8 & 0.170222 & 148.86660& 152440\\ \hline
	 12 & $(5^2,7)$&2 & 4 & 1 & 3 & 0.219683 & 47.520125& 1521\\ \hline
	 13 & $(5^2,9)$&6 & 5 & 1 & 5 & 0.421578 & 35.208505& 4507\\ \hline
	14 & $(5^2,11)$&2 & 5 & 1 & 3 & 0.176146 & 70.124930& 2244\\ \hline
	15 & $(5^3,6)$&6 & 5 & 1 & 4 & 0.525578 & 26.734639& 3423\\ \hline
	16 & $(5^5,6)$&6 & 9 & 10 & 4 & 0.390055 & 55.838482& 7148\\ \hline
	17 & $(5,15)$&2 & 5 & 1 & 2 & 0.473298 & 25.241167& 808\\ \hline
	18 & $(5,40)$&6 & 9 & $x^2+\beta^3x+\beta$ & 4 & 0.088640 & 238.91192& 61162\\ \hline
	19 & $(5^3,8)$&6 & 6 & 1 & 6 & 0.454072 & 39.438940& 5049\\ \hline
	20 & $(5,16)$&6 & 4 & $x+1$ & 7 & 0.038742 & 363.35624& 46510\\ \hline
	21 & $(5,24)$&6 & 6 & $x^4-1$ & 10 & 0.086200 & 245.61740& 251513\\ \hline

	\end{tabular}
\end{center}
Next, we refer to Table 2 to note that Theorem \ref{PSC} holds for the pairs $(5, 11)$, $(5, 13)$, $(5, 14)$, $(5, 15)$, $(5, 16)$,  $(5, 17)$, $(5, 18)$, $(5, 19)$, $(5, 21)$, $(5, 22)$, $(5, 24)$, $(5, 27)$, $(5, 30)$, $(5, 36)$, $(5, 40)$, $(5^2, 7)$, $(5^2, 9)$, $(5^2, 11)$, $(5^3, 6)$, $(5^3, 8)$, $(5^5, 6)$. Thus, only left \textbf{possible exceptions} in the case $m'\nmid q-1$ are $(5, 6)$,$(5, 7)$, $(5,8)$, $(5,9)$, and  $(5, 12).$\\\\
\subsection{Part 2.} In this  part we shall consider $m=3, 4.$ Following result will be  required for further calculation, which follows on the lines of \cite[Lemma 51]{AnjuCohen}.
\begin{lemma}\label{boundlemma1}
	Let $k\in \mathbb{N}$ such that $\omega(k)\geq 2828$. Then $W(k)< k^\frac{1}{13}.$
\end{lemma}
Also, $W(x^m-1)\leq 16$. Now, first assume $\omega(q^m-1)\geq 2828$, then (\ref{main}) and Lemma \ref{boundlemma1} together implies $(q, m)\in T_2$ if $q^{\frac{m}{2}-1}>64\cdot q^{\frac{2m}{13}}$ i.e., $q^{\frac{9m}{26}-1}>64$ or $q^m>64^{\frac{26m}{9m-26}}$, sufficient if $q^m>64^{78}$, which is  true for $\omega(q^m-1)\geq 2828$. To make further progress we follow \cite{COMM}.  Next, assume $88\leq \omega(q^m-1)\leq 2827$. In Theorem \ref{PSC}, let $g=x^m-1$ and $l$ to be the product of least $88$ primes dividing
$q^m-1$ i.e., $W(l) = 2^{88}$. Then $r\leq 2739$ and $\delta $ will be at least its value
when $\{p_1, p_2, \cdots, p_{2739}\}=\{461, 463, \cdots, 25667\}$.   This gives $\delta> 0.0041806$
and $\Delta<1.3101\times 10^6$, hence $4\Delta W(g) W(l)^2 < 8.0309 \times 10^{60} = R$ (say). By
Theorem \ref{PSC} $(q,m) \in T_2$ if $q^{\frac{m}{2}-1} > R$ or $q^m > R^{\frac{2m}{m-2}}$.
But $m\geq 3$ implies $\frac{2m}{m-2}\leq 6$. Therefore, if $q^m > R^6$
or $q^m > 2.6828\times 10^{365}$
then $(q,m)\in T_2$. Hence, $\omega(q^m-1)\geq  152$ gives $(q,m)\in T_2$. Repeating this process of Theorem \ref{PSC} for the values in Table 3 implies $(q, m)\in T_2$ if $q^{\frac{m}{2}-1}>889903387$. Thus, for $m=3$ it is sufficient if $q>(889903387)^2$ and for $m=4$ we need $q>889903387$. Hence, only possible exceptions are $(5, 3), (5^2, 3), \cdots, (5^{25}, 3)$ and $(5, 4), (5^2, 4), \cdots, (5^{12}, 4)$. However, Table 4 implies that Theorem \ref{PSC} holds for $(5^9, 3), (5^{11}, 3), (5^{12}, 3), (5^{13}, 3), \cdots, (5^{25}, 3)$ and $(5^6, 4), (5^7, 4), \cdots, (5^{12}, 4)$. Thus, only \textbf{possible exceptions} here are $(5, 3), (5^2, 3), \cdots, (5^8, 3)$ and $(5^{10}, 3)$, and $(5, 4), (5^2, 4), \cdots, (5^5, 4)$.
\begin{center}
	$$\text{Table 3}$$
	\begin{tabular}{|m{.6cm}|m{3.2cm}|m{.7cm}|m{2cm}|m{2.2cm}|m{2.3cm} |}
		\hline
		Sr. No.    &  $a \leq \omega(q^m-1)\leq b$ & $W(l)$   &  $\delta>$   & $\Delta<$ & $4 \Delta W(g)$ $ W(l)^2$ $<$ \\
		\hline
		1      &  $a=17,~~ b=151$ & $2^{17}$ & $0.0347407$  & $7687.5008$ & $8.4526\times 10^{15}$  \\
		2      &$a=9,~~ b=51$  & $2^9$ & $0.0550187$   & $1510.5788$ & $2.5344\times 10^{10}$  \\
		3      &$a=7,~~ b=37$  & $2^7$ & $0.0064402$   & $9163.1796$ & $9608289244$  \\
			4      &$a=7,~~ b=36$  & $2^7$ & $0.0191790$   & $2973.9903$ & $3118453847$  \\
				5      &$a=7,~~ b=34$  & $2^7$ & $0.0458469$   & $1158.0218$ & $1214272852$  \\
				6      &$a=7,~~ b=33$  & $2^7$ & $0.0602354$   & $848.6790$ & $889903387$  \\	
			\hline	
	\end{tabular}
\end{center}

\begin{center}
	Table 4
	\begin{tabular}{|m{.6cm}|m{1.3cm}|m{.6cm}|m{.6cm}|m{.6cm}|m{.6cm} |m{1.5cm}|m{1.7cm}|m{1.8cm}|}
		\hline 		Sr. No.    &  $(q, m)$ & $l$   &  $r$   & $g$ & $s$ & $\delta>$& $\Delta< $ & $4\Delta W(g) $ $W(l)^2<$ \\
		\hline
	1 & $(5^9,3)$&2 & 7 & 1 & 2 & 0.801533 & 20.714128& 663\\ \hline
	2 & $(5^{11},3)$&2 & 4 & 1 & 2 & 0.925433 & 11.725177& 376\\ \hline
	3 & $(5^{12},3)$&6 & 9 & 1 & 3 & 0.330478 & 62.518314& 8003\\ \hline
4 & $(5^{13},3)$&2 & 4 & 1 & 2 & 0.910167 & 11.888295& 381\\ \hline
5 & $(5^{14},3)$&6 & 10 & 1 & 3 & 0.508443 & 45.269297& 5795\\ \hline
	6 & $(5^{15},3)$&2 & 10 & 1 & 2 & 0.603902 & 36.773815& 1177\\ \hline
	7 & $(5^{16},3)$&6 & 9 & 1 & 3 & 0.368379 & 56.291827& 7206\\ \hline
	
	8 & $(5^{17},3)$&2 & 6 & 1 & 2 & 0.930565 & 15.970005& 512\\ \hline

	9 & $(5^{18},3)$&6 & 12 & 1& 3 & 0.499055 & 54.098369& 6925\\ \hline

	10 & $(5^{19},3)$&2 & 5 & 1 & 2 & 0.924693 & 13.895837& 445\\ \hline
	11 & $(5^{20},3)$&6 & 15 & 1 & 3 & 0.183646 & 176.24807& 22560\\ \hline
	12& $(5^{21},3)$&2 & 9 & 1 & 2 & 0.822416& 25.102645& 804\\ \hline
	13 & $(5^{22},3)$&6 & 10 & 1 & 3 & 0.522529 & 44.102865& 5646\\ \hline
14& $(5^{23},3)$&2 & 7 & 1 & 2 & 0.920550 & 18.294603& 586\\ \hline
15 & $(5^{24},3)$&6 & 14 & 1 & 3 & 0.296682 & 103.11815& 13200\\ \hline
16 & $(5^{25},3)$&2 & 14 & 1 & 2 & 0.666688 & 45.498589& 1456\\ \hline
	17 & $(5^6,4)$&6 & 6 & 1 & 4 & 0.485944 & 32.867712& 4208\\ \hline
18 & $(5^7,4)$&2 & 6 & 1 & 4 & 0.105913 & 143.62473& 4596\\ \hline
 19 & $(5^8,4)$&2 & 7 & 1 & 4 & 0.054494 & 313.95724& 10047\\ \hline
20 & $(5^9,4)$&6 & 9 & 1 & 4 & 0.330476 & 65.544620& 8390\\ \hline
21 & $(5^10,4)$&6 & 9 & 1 & 4 & 0.568640 & 38.930216& 4984\\ \hline
22 & $(5^11,4)$&2 & 8 & 1 & 4 & 0.039829 & 479.03888& 15330\\ \hline
23 & $(5^12,4)$&6 & 9 & 1 & 4 & 0.368379 & 59.006421& 7553\\ \hline

	\end{tabular}
\end{center}
Further, for all the left \textbf{possible exceptions} we checked Theorem \ref{MPSC} and got it verified in case of $(5^7, 3), (5^5, 4)$ and $(5, 9)$ for the values in Table 5. 
\begin{center}
	Table 5
	\begin{tabular}{|m{.4cm}|m{1cm}|m{.3cm}|m{1.7cm}|m{1.1cm}|m{.8cm} |m{1.8cm}|m{1.7cm}|m{1cm}|}
		\hline 		Sr. No.    &  $(q, m)$ & $k$   &  $P$   & $L$ & $f$ & $G$& $H $ & $R'<$ \\
		\hline
		1 & $(5,9)$&2 &589 & 829 & $x-1$ & $x^2+x+1$ & $x^6+x^3+1$& 269\\ \hline
		2 & $(5^7,3)$&2 &229469719 & 519499 & $x-1$ & 1 & $x^2+x+1$& 262\\ \hline
		3 & $(5^9,4)$&6 &216878233 & 9161 & $x+1$ & $x^2+x+\beta^3$ & $x+\beta^3$& 2788\\ \hline

	\end{tabular}
\end{center}
Where, $R'$ represent the right hand side value of (\ref{mpsc}).
Hence, all the results from part 1 and part 2 collectively implies Theorem \ref{example}.

       \bibliographystyle{plain}
	\bibliography{Prescribednormal}

\begin{thebibliography}{10}

\bibitem{application}
G.~B. Agnew, R.~C. Mullin, I.~M. Onyszchuk, and S.~A. Vanstone.
\newblock An implementation for a fast public-key cryptosystem.
\newblock {\em J. Cryptology}, 3(2):63--79, 1991.

\bibitem{special}
Anju and R.~K. Sharma.
\newblock Existence of some special primitive normal elements over finite
  fields.
\newblock {\em Finite Fields Appl.}, 46:280--303, 2017.

\bibitem{booker}
A.~Booker, S.~D. Cohen, N.~Sutherland, and T.~Trudgian.
\newblock Primitive values of quadratic polynomials in a finite field.
\newblock {\em Math. Comp.}, 88(318):1903--1912, 2019.

\bibitem{chou}
W.~S. Chou and S.~D. Cohen.
\newblock Primitive elements with zero traces.
\newblock {\em Finite Fields Appl.}, 7(1):125--141, 2001.

\bibitem{even}
S.~D. Cohen.
\newblock Pairs of primitive elements in fields of even order.
\newblock {\em Finite Fields Appl.}, 28:22--42, 2014.

\bibitem{AnjuCohen}
S.~D. Cohen and A.~Gupta.
\newblock Primitive element pairs with a prescribed trace in the quartic
  extension of a finite field.
\newblock {\em J. Algebra Appl.}, 2020, DOI:
  https://doi.org/10.1142/S0219498821501681.

\bibitem{pnbtwc}
S.~D. Cohen and S.~Huczynska.
\newblock The primitive normal basis theorem--without a computer.
\newblock {\em Lond. Math. Soc.}, 67(2):41--56, 2003.

\bibitem{JNT}
S.~D. Cohen, H.~Sharma, and R.~Sharma.
\newblock Primitive values of rational functions at primitive elements of a
  finite field.
\newblock {\em J. Number Theory}, 219:237--246, 2021.

\bibitem{PPR4}
L.~Fu and D.~Wan.
\newblock A class of incomplete character sums.
\newblock {\em Q. J. Math.}, (4):1195--1211, 2018.

\bibitem{FFAAnju}
A.~Gupta, R.~K. Sharma, and S.~D. Cohen.
\newblock Primitive element pairs with one prescribed trace over a finite
  field.
\newblock {\em Finite Fields Appl.}, 54:1--14, 2018.

\bibitem{pnbt}
H.~W. Lenstra~Jr. and R.~J. Schoof.
\newblock Primitive normal bases for finite fields.
\newblock {\em Math. Comp.}, 48(177):217--231, 1987.

\bibitem{lidl}
R.~Lidl and H.~Niederreiter.
\newblock {\em Finite fields}, volume~20.
\newblock Cambridge Univ. Press, Cambridge (UK), 1997.

\bibitem{COMM}
H.~Sharma and R.~K. Sharma.
\newblock Existence of primitive pairs with prescribed traces over finite
  fields.
\newblock {\em Comm. Algebra}, 2020, DOI:
  https://doi.org/10.1080/00927872.2020.1852243.

\bibitem{ambrish}
R.~K. Sharma, A.~Awasthi, and A.~Gupta.
\newblock Existence of pair of primitive elements over finite fields of
  characteristic 2.
\newblock {\em J. Number Theory}, 193:386--394, 2018.

\bibitem{CommAnju}
R.~K. Sharma and A.~Gupta.
\newblock Pair of primitive elements with prescribed traces over finite fields.
\newblock {\em Comm. Algebra}, 47:1278--1286, 2017.

\bibitem{shuqin2004character}
F.~Shuqin and H.~Wenbao.
\newblock Character sums over galois rings and primitive polynomials over
  finite fields.
\newblock {\em Finite Fields Appl.}, 10(1):36--52, 2004.

\end{thebibliography}
\end{document}